\documentclass[a4paper, 10pt, parskip=half]{scrartcl}
\usepackage[utf8]{inputenc}
\usepackage[T1]{fontenc}
\usepackage{lmodern}
\usepackage{amssymb}
\usepackage{mathrsfs}
\usepackage{amsmath,amssymb,amsthm}
\usepackage{dsfont}
\usepackage{mathtools}
\usepackage{xcolor}
\usepackage{hyperref}
\hypersetup{
  colorlinks=true,
  linkcolor=black,
  citecolor=black,
  urlcolor=blue, 
  pdftitle={When do Keller--Segel systems with heterogeneous logistic sources admit generalized solutions?},
  pdfauthor={Jianlu Yan, Mario Fuest},
  pdfkeywords={chemotaxis; logistic source; generalized solution},
  bookmarksopen=true,
}
\RequirePackage{geometry}
\newtheorem{thm}{Theorem}[section]

\newtheorem{lem}[thm]{Lemma}

\newtheorem{rem}[thm]{Remark}
\theoremstyle{definition}
\newtheorem{defn}[thm]{Definition}
\numberwithin{equation}{section}
\newcommand{\be}{\begin{equation}}
\newcommand{\ee}{\end{equation}}

\newcommand{\R}{\mathbb{R}}

\newcommand{\N}{\mathbb{N}}

\newcommand{\ur}[1]{\mathrm{#1}}
\newcommand{\ure}{\ur e}

\newcommand{\eps}{\varepsilon}

\newcommand{\gt}{>}
\newcommand{\lt}{<}

\DeclareMathOperator{\supp}{supp}

\newcommand{\defs}{\coloneqq}
\newcommand{\sfed}{\eqqcolon}

\newcommand{\ra}{\rightarrow}

\newcommand{\nea}{\nearrow}

\newcommand{\sea}{\searrow}

\newcommand{\ol}{\overline}

\newcommand{\loc}{\mathrm{loc}}

\newcommand{\embed}{\hookrightarrow}

\newcommand{\hp}{\hphantom}
\newcommand{\pe}{\mathrel{\hp{=}}}

\newcommand{\tmaxe}{T_{\max, \eps}}

\newcommand{\intom}{\int_\Omega}

\newcommand{\Ombar}{\ol \Omega}

\newcommand{\leb}[2][\spaceused]{\ensuremath{L^{#2}(#1)}}
\newcommand{\sob}[3][\spaceused]{\ensuremath{W^{#2, #3}(#1)}}

\newcommand{\con}[2][\ol \spaceused]{\ensuremath{C^{#2}(#1)}}

\newcommand{\dx}{\,\mathrm{d}x}
\newcommand{\ds}{\,\mathrm{d}s}
\newcommand{\dt}{\,\mathrm{d}t}

\newcommand{\ddt}{\frac{\mathrm{d}}{\mathrm{d}t}}

\newcommand{\ue}{u_\eps}
\newcommand{\uep}{u_{\eps^\prime}}

\newcommand{\ve}{v_\eps}
\newcommand{\vep}{v_{\eps^\prime}}

\newcommand{\qsup}{q_{\sup}}
\newcommand{\rsup}{r_{\sup}}

\makeatletter
\renewenvironment{proof}[1][\proofname]{%
  \parskip=0pt \par
  \pushQED{\qed}%
  \normalfont \topsep0\p@\relax
  \trivlist
  \item[\hskip\labelsep\scshape
  #1\@addpunct{.}]\ignorespaces
}{%
  \popQED\endtrivlist\@endpefalse
}
\makeatother

\setkomafont{title}{\normalfont \Large}
\title
{When do Keller--Segel systems with heterogeneous logistic sources admit generalized solutions?}
\usepackage{authblk}

\author[1,2]{Jianlu Yan\footnote{e-mail:\ 230159430@seu.edu.cn}}
\author[2]{Mario Fuest\footnote{e-mail:\ fuestm@math.uni-paderborn.de}}
\affil[1]{Institute for Applied Mathematics, School of Mathematics, Southeast University, Nanjing 211189, P.~R.~China}
\affil[2]{Institut für Mathematik, Universität Paderborn, Warburger Str.~100, 33098 Paderborn, Germany}  
\begin{document}
\date{}
\maketitle

\KOMAoptions{abstract=true}
\begin{abstract}
  \noindent
  We construct global generalized solutions to the chemotaxis system
  \[
    \begin{cases}
      u_t = \Delta u - \nabla \cdot (u \nabla v) + \lambda(x) u - \mu(x) u^\kappa, \\
      v_t = \Delta v - v + u
    \end{cases}
  \]
  in smooth, bounded domains $\Omega \subset \mathbb R^n$, $n \geq 2$,
  for certain choices of $\lambda, \mu$ and $\kappa$. \\[0.5pt]
  Here, inter alia, the selections $\mu(x) = |x|^\alpha$ with $\alpha < 2$ and $\kappa = 2$
  as well as $\mu \equiv \mu_1 > 0$ and $\kappa > \min\{\frac{2n-2}{n}, \frac{2n+4}{n+4}\}$
  are admissible (in both cases for any sufficiently smooth $\lambda$). \\[0.5pt]
  While the former case appears to be novel in general, in the two- and three-dimensional setting,
  the latter improves on a recent result by Winkler (Adv.\ Nonlinear Anal.\ \textbf{9} (2019), no.\ 1, 526–566),
  where the condition $\kappa > \frac{2n+4}{n+4}$ has been imposed.
  In particular, for $n = 2$, our result shows that taking \emph{any} $\kappa > 1$
  suffices to exclude the possibility of collapse into a persistent Dirac distribution.\\[0.5pt]
\end{abstract}

\section{Introduction}
After the seminal work by Keller and Segel \cite{KellerSegelTravelingBandsChemotactic1971} nearly half a century ago,
biologists and mathematicians alike have shown great interest
in various systems describing chemotaxis,
i.e.\ the partially directed movement of (inter alia) cells towards higher concentration of a chemical substance,
see for instance \cite{BellomoEtAlMathematicalTheoryKeller2015} for an overview.

As discussed in the recent survey \cite{LankeitWinklerFacingLowRegularity2019},
many taxis systems lack sufficient regularity to obtain global classical solutions and hence often one has to resort to certain weaker solution concepts.
In the present paper, we will construct global generalized solutions of the initial boundary value problem
\begin{eqnarray}\label{e11}
\begin{cases}
\begin{array}{lll}
\medskip
u_t =\Delta u- \nabla\cdot(u\nabla v)+\lambda(x) u - \mu(x) u^\kappa,&x\in\Omega,\ t>0,\\
\medskip
v_t =\Delta v -v+u,&x\in\Omega,\ t>0,\\
\medskip
\partial_\nu u = 0,\ \partial_\nu v = 0,&x\in\partial\Omega,\ t>0,\\
\medskip
u(x,0)=u_0(x),\ v(x,0)=v_0(x),&x\in\Omega,
\end{array}
\end{cases}
\end{eqnarray}
where $\Omega\subset \mathbb{R}^n$, $n\geq1$, is a bounded domain with smooth boundary,
$\kappa \gt 1$ is a parameter
and $\lambda$, $\mu$ as well as $u_0, v_0$ are sufficiently regular given functions.
Here, $u$ denotes the density of the cells and $v$ represents the concentration of the chemical signal.
Both undergo random motion (terms $\Delta u$ and $\Delta v$),
the cells undergo logistic-type growth $(+ \lambda(x) u - \mu(x) u^\kappa$),
the chemical signal is produced by the cells ($+ u$) and decays exponentially ($-v$)
and, most importantly, the cells are attracted by a higher chemical concentration
($-\nabla \cdot (u \nabla v)$, the so-called chemotaxis term).

The system~\eqref{e11} forms also the basis of more complex models describing
population dynamics~\cite{HillenPainterUserGuidePDE2009, ShigesadaEtAlSpatialSegregationInteracting1979},
pattern formation~\cite{WoodwardEtAlSpatiotemporalPatternsGenerated1995} or
cancer invasion processes~\cite{ChaplainLolasMathematicalModellingCancer2005}, to just name a few examples.
A recent overview on several chemotaxis systems, many of which feature a logistic source term,
is given in~\cite{PainterMathematicalModelsChemotaxis2019}.

Lately, also space-depending functions $\lambda$, $\mu$, reflecting heterogeneous environments,
have been considered in parabolic--elliptic versions of~\eqref{e11}
both for $\Omega = \R^n$~\cite{SalakoShenParabolicellipticChemotaxisModel2018, SalakoShenParabolicellipticChemotaxisModel2018a, SalakoShenParabolicellipticChemotaxisModel2018b}
and for bounded domains~\cite{FuestFinitetimeBlowupTwodimensional2020}.
We refer especially to the introduction of the latter article for a more detailed motivation.

Moreover, for the effect of nonlinear degradation compared to additional modifications,
such as nonlinear chemotactic sensitivity or nonlinear signal production,
we refer to the recent work~\cite{LiFullyParabolicChemotaxis2019} (and the references therein),
where these amendments have been studied together.

Let us now briefly summarize the findings on global existence regarding the system~\eqref{e11}.
For constant $\lambda, \mu \gt 0$ and $\kappa = 2$,
global classical solutions to \eqref{e11} are known to exist
if either $n = 2$ and $\mu$ is merely assumed to be positive \cite{OsakiEtAlExponentialAttractorChemotaxisgrowth2002}
or if $n \ge 3$ and $\mu$ is sufficiently large \cite{WinklerBoundednessHigherdimensionalParabolicparabolic2010}.
These results already show a certain relaxing effect of quadratic degradation terms.
After all, in the absence of cell proliferation, that is, for $\lambda \equiv 0$ and $\mu \equiv 0$,
solutions blowing up in finite time are known to exist
both in two~\cite{HorstmannWangBlowupChemotaxisModel2001, SenbaSuzukiParabolicSystemChemotaxis2001}
and higher~\cite{WinklerFinitetimeBlowupHigherdimensional2013} dimensions.

However, not every superlinear dampening term guarantees the existence of global classical solutions.
In certain parabolic--elliptic simplifications of \eqref{e11},
for constant $\lambda, \mu \gt 0$ and sufficiently small $\kappa \gt 1$,
solutions blowing up in finite time have been constructed,
at first only for $n \ge 5$~\cite{WinklerBlowupHigherdimensionalChemotaxis2011}
and then also for $n \ge 3$~\cite{WinklerFinitetimeBlowupLowdimensional2018}.
Corresponding results have recently also been obtained for heterogeneous environments
in two-~\cite{FuestFinitetimeBlowupTwodimensional2020}
and higher~\cite{BlackEtAlRelaxedParameterConditions2020} dimensional settings.
Still, none of these works claim that the conditions on ($\mu$ and) $\kappa$ are optimal.
In fact, for quite a large range of parameters,
it still appears to be unknown whether global classical solutions exist for all suitably smooth initial data.

Thus, in certain cases, one might need to resort to more general solution concepts.
The most prominent result in this direction is probably~\cite{LankeitEventualSmoothnessAsymptotics2015},
where weak solutions have been constructed for constant $\lambda, \mu \gt 0$ and $\kappa = 2$.
In the same work, it is also shown that in three dimensional convex domains and if $\lambda$ is sufficiently small compared to $\mu$,
these solutions even become smooth after some time.

More recently, again for constant $\lambda, \mu \gt 0$ and $\kappa = 2$,
both for~\eqref{e11} (in the two dimensional setting)~\cite{LankeitImmediateSmoothingGlobal2020}
and for parabolic--elliptic versions thereof~\cite{WinklerHowStrongSingularities2019},
it has been analyzed for which initial data one can construct solutions becoming smooth instantaneously.

Moreover, for constant $\lambda, \mu \gt 0$
the question how large $\kappa$ needs to be for~\eqref{e11} to admit at least global generalized solutions has also been asked.
A first partial answer was already given in 2008:
For a parabolic--elliptic simplification of~\eqref{e11},
the condition $\kappa \gt 2-\frac1n$ suffices~\cite{WinklerChemotaxisLogisticSource2008}.
Meanwhile it it known that the same results also holds also in the fully parabolic case~\cite{ViglialoroVeryWeakGlobal2016}.
Recently, this condition has been improved to $\kappa \gt \frac{2n+4}{n+4}$~\cite{WinklerRoleSuperlinearDamping2019}.

While the solution concepts in the articles above differ,
they have in common that they exclude the collapse into a persistent Dirac-type distribution.
The latter has been observed for (a simplified version of) the system with no proliferation $(\lambda \equiv 0$, $\mu \equiv 0)$
in the two-dimensional setting~\cite{BilerRadiallySymmetricSolutions2008}
and is certainly one of the most striking features of chemotaxis systems.

\paragraph{Main results.}
In the present paper, we substantially extend the set of superlinear degrading terms
which are known to allow for generalized solutions
excluding the possibility of collapse into persistent Dirac-type distributions.
Our purpose is two-fold:
On the one hand, we are interested in the general interplay between the space-dependent function $\mu$,
which might vanish at some points,
and the superlinear degrading term $-u^\kappa$, $\kappa \gt 1$.
On the other hand, for constant $\mu$, we improve on the conditions imposed on $\kappa$
in~\cite{ViglialoroVeryWeakGlobal2016} and~\cite{WinklerRoleSuperlinearDamping2019}.

Our main result is
\begin{thm}\label{main_result}
  Let $\Omega\subset\mathbb{R}^n$, $n \ge 2$, be a bounded domain with smooth boundary, $\beta \in (0, 1)$, $s \gt 0$ and
  \begin{align}\label{cond:kappa}
        \kappa
    \gt \max\left\{
          \frac{2n(s+1)}{(n+2) s},
          \min\left\{
            \frac{2(n-1) s + n}{n s}, 
            \frac{2(n+2) s + 2n}{(n+4) s}
          \right\}
        \right\}.
  \end{align}
  Suppose moreover that $\lambda, \mu \in \con\beta$ fulfill $\mu \ge 0$ and
  \begin{align}\label{cond:mu}
    \intom \mu^{-s} \lt \infty.
  \end{align}

  Then for all
  \begin{equation}\label{e12}
  0 \le u_0\in \leb1
  \quad \text{and} \quad
  0 \le v_0\in \leb\infty,
  \end{equation}
  the system~\eqref{e11} possesses at least one global generalized solution in the sense of Definition~\ref{defi21} below.
\end{thm}

\begin{rem}
  The restriction $n \ge 2$ in Theorem~\ref{main_result} is not needed.
  After all, even without any degrading term (i.e.\ $\mu \equiv 0$),
  for $n = 1$ global classical solutions exist~\cite{OsakiYagiFiniteDimensionalAttractor2001}.
  However, the condition $n \ge 2$ allows for a briefer reasoning in some places.
\end{rem}

\begin{rem}\label{rm:proto}
  Let $\Omega \subset \R^n$, $n \ge 2$, be a smooth, bounded domain containing $0$
  and take as a prototypical example $\mu(x) = \mu_1 |x|^\alpha$, $x \in \Ombar$, for $\mu_1 \gt 0$ and $\alpha \ge 0$.
  Then~\eqref{cond:mu} is fulfilled for all $s \in (0, \frac{n}{\alpha})$
  and if
  \begin{align}\label{cond:kappa_proto}
        \kappa
    \gt \max\left\{
          \frac{2n + 2 \alpha}{n+2},
          \min\left\{
            \frac{2(n-1) + \alpha}{n}, 
            \frac{2(n+2) + 2 \alpha}{n+4}
          \right\}
        \right\}
  \end{align}
  or, equivalently, ($\kappa \gt \min\left\{ \frac{2n-2}{n}, \frac{2n+4}{n+4} \right\}$ and)
  \begin{align}\label{cond:mu_proto}
        \alpha
    \lt \min\left\{
          \frac{(\kappa - 2)n + 2\kappa}{2},
          \max\left\{
            (\kappa - 2)n + 2,
            \frac{(\kappa - 2)n + 4 \kappa - 4}{2}
          \right\}
        \right\}
  \end{align}
  holds, then \eqref{cond:kappa} is also satisfied for some sufficiently large $s \in (0, \frac{n}{\alpha})$.
  Thus, in this case Theorem~\ref{main_result} provides the existence of global generalized solutions.

  If $\kappa = 2$, then \eqref{cond:mu_proto} reduces to the requirement $\alpha \lt 2$.
  Furthermore, thanks to the assumption that $n \ge 2$, for $\alpha = 0$ the condition \eqref{cond:kappa_proto} becomes
  \begin{align}\label{cond:kappa_alpha=0}
        \kappa
    \gt \min\left\{ \frac{2n-2}{n}, \frac{2n+4}{n+4} \right\}.
  \end{align}
  Obviously, this improves on the condition $\kappa \gt \frac{2n-1}{n}$ taken in~\cite{ViglialoroVeryWeakGlobal2016}.
  (Note that, admittedly, the solution concept taken there while similar to ours is slightly stronger.)
  Moreover, as $\frac{2n-2}{n} \lt \frac{2n+4}{n+4}$ if and only if $n \lt 4$,
  for the physically relevant space dimensions $n = 2$ and $n = 3$,
  \eqref{cond:kappa_alpha=0} is a weaker assumption than $\kappa \gt \frac{2n+4}{n+4}$
  which has been imposed in~\cite{WinklerRoleSuperlinearDamping2019}.
\end{rem}

\begin{rem}
  In the prototypical example in Remark~\ref{rm:proto}, the function $\mu$ has (at most) one zero
  and, with obvious modifications, similar results hold when $\mu$ is allowed to have a finite number of roots.
  However, we would like to mention that Theorem~\ref{main_result} is also applicable
  for certain $\mu$ vanishing on some null sets with infinitely many points,
  for instance on lower dimensional manifolds such as line segments or circles.
\end{rem}

\paragraph{Main ideas.}
Our definition of generalized solutions follows~\cite{WinklerLargeDataGlobalGeneralized2015};
that is, $v$ is required to be a weak solution and $u$ has to be both a `mass subsolution' and a `logarithmic supersolution'.
This concept, which is consistent with that of classical solutions, is introduced in more detail in Section~\ref{sec:sol_concept}.

The proof then mainly consists in obtaining sufficiently strong a priori estimates
for solutions to certain approximative systems (see~\eqref{e210} below)
which allow for the application of various compactness theorems.

As a first step, we make use of the logistic term in the first equation in~\eqref{e11} and condition~\eqref{cond:mu}
to obtain an $L_{\loc}^\kappa$-$L^p$ bound for $u$ for a certain $p \gt 1$ in Lemma~\ref{lm:u_l_kappa_l_p},
which in turn directly implies boundedness in $L_{\loc}^p(\Ombar \times [0, \infty))$.
This latter information will turn out to be crucial to pass to the limit in the second equation.

Next, relying on parabolic regularity theory, we want to derive a uniform $L^r$ bound for $v$.
To that end, we have (at least) two possibilities:
We could either make use of the local-in-time mass boundedness
(which is readily obtained upon integrating the first equation in~\eqref{e11}, see Lemma~\ref{lem2.2})
or of the aforementioned space-time bound for $u$.
As it turns out, both options have their merit---it depends on the choice of parameters which one is to be preferred.
In fact, if $\lambda, \mu \gt 0$ are constant and $n \in \{2, 3\}$, the former method turns out to be stronger.
This is the reason why we are able to improve on the corresponding result in~\cite{WinklerRoleSuperlinearDamping2019},
where only the latter method has been employed---which in turn is more powerful for (constant $\lambda, \mu \gt 0$ and) $n \gt 4$
and equally strong for $n = 4$.
Both these options are explored together in Lemma~\ref{lemv}.

These estimates combined then imply, precisely due to the condition~\eqref{cond:mu},
bounds of $\nabla v$ in $L_{\loc}^2(\Ombar \times [0, \infty))$
and of $uv$ in $L_{\loc}^\gamma(\Ombar \times [0, \infty))$ for some $\gamma \gt 1$, see Lemma~\ref{lem2.3}.
While the former is crucial for estimating $\ddt \intom \ln (u + 1)$,
and hence for obtaining a bound for $\int_0^T \intom \frac{|\nabla u|^2}{(u+1)^2}$ (cf.\ Lemma~\ref{lem2.4}),
the latter allows us to make use of an energy identity associated with the second equation in~\eqref{e11}
to obtain even strong $L^2$~convergence of $\nabla v$ in Lemma~\ref{lem2.9}.

Finally, at the end of Section~\ref{sec:proof}, we combine the information gathered and prove Theorem~\ref{main_result}.

\section{A generalized solution concept and approximate solutions}\label{sec:sol_concept}
Throughout the sequel, we fix $n \ge 2$, a smooth, bounded domain $\Omega \subset \R^n$,
$s \gt 0$, $\kappa$ satisfying \eqref{cond:kappa},
$\beta \in (0, 1)$, $\lambda, \mu \in \con\beta$ with $\mu \ge 0$ fulfilling \eqref{cond:mu}
and $u_0, v_0$ satisfying~\eqref{e12}.
Additionally, for $c \gt 0$, we always set $\frac{c}{0} \defs \infty$ and $\frac{c}{\infty} \defs 0$.

\begin{defn}\label{defi21}
A pair
\begin{equation}\label{e21}
  (u, v) \in L^1_{\loc}(\Ombar\times[0,\infty)) \times L^2_{\loc}([0,\infty);W^{1,2}(\Omega))
\end{equation}
with
\begin{equation}\label{e22}
 u \ge 0, v \ge 0 \text{ a.e.~in } \Omega \times (0,\infty)
\end{equation}
as well as
\begin{equation}\label{e23}
\nabla \ln(u+1)\in L^2_{\loc}(\Ombar\times[0,\infty))\end{equation}
is called a \emph{global generalized solution }of $(\ref{e11})$ if $u$ has the property
\begin{equation}\label{e24}
  \int_{\Omega} u(\cdot, T)-\int_{\Omega} u_{0}
  \leq\int_{0}^{T} \int_{\Omega} \lambda u
  -\int_{0}^{T} \int_{\Omega} \mu u^\kappa
\end{equation}
for a.e.\ $T>0$,
\begin{equation}\label{e25}
\begin{split}
&\pe-\int_{0}^\infty\int_{\Omega} \ln (u+1) \varphi_{t}-\int_{\Omega}\ln \left(u_{0}+1\right) \varphi(\cdot, 0) \\ 
&\geq \int_{0}^{\infty} \int_{\Omega}|\nabla \ln (u+1)|^{2} \varphi-\int_{0}^{\infty} \int_{\Omega} \nabla \ln (u+1) \cdot \nabla \varphi \\ 
&\pe-\int_{0}^{\infty} \int_{\Omega} \frac{u}{u+1}(\nabla \ln (u+1) \cdot \nabla v) \varphi+\int_{0}^{\infty} \int_{\Omega} \frac{u}{u+1} \nabla v \cdot \nabla \varphi\\
&\pe+\int_{0}^{\infty} \int_{\Omega} \frac{\lambda u}{u+1}\varphi-\int_{0}^{\infty} \int_{\Omega} \frac{\mu u^\kappa}{u+1}\varphi
\end{split}
\end{equation}
holds for all nonnegative $\varphi\in C_{c}^{\infty} (\Ombar\times [0,\infty))$
and
\begin{equation}\label{e26}
-\int_{0}^{\infty}\int_{\Omega}v\varphi_t
-\int_{\Omega}v_0\varphi(\cdot,0)
=-\int_{0}^{\infty}\int_{\Omega}\nabla v\cdot\nabla\varphi
-\int_{0}^{\infty}\int_{\Omega}v\varphi
+\int_{0}^{\infty}\int_{\Omega}u\varphi
\end{equation}
holds for all $\varphi\in C_{c}^{\infty} (\Ombar\times [0,\infty))$.
\end{defn}

\begin{rem}
  That any global classical solution $(u, v)$ of~\eqref{e11}
  is also a global generalized solution (even with equality in~\eqref{e25}),
  can be seen by a direct computation.

  Moreover, the above definition is consistent with classical solutions.
  That is, if $(u, v) \in (C^0(\Ombar \times [0, \infty)) \cap C^{2, 1}(\Ombar \times (0, \infty)))^2$ is a global generalized solution,
  then $(u, v)$ is a also a (global) classical solution.
  We refer to \cite[Lemma~2.2]{HeihoffGeneralizedSolutionsSystem2019} and \cite[Lemma~2.1]{WinklerLargeDataGlobalGeneralized2015}
  for corresponding proofs in closely related settings.
\end{rem}

In order to construct such generalized solutions by an approximation procedure,
we henceforth fix families $(u_{\eps 0})_{\eps \in (0, 1)}, (v_{\eps 0})_{\eps \in (0, 1)} \subset \con\infty$
with
\begin{align}\label{eq:conv_init}
  0 \le u_{\eps 0} \ra u_0 \quad \text{in $\leb1$}
  \quad \text{and} \quad
  0 \le v_{\eps 0} \ra v_0 \quad \text{in $\leb r$ for all $r \in [1, \infty)$}
  \qquad \text{as $\eps \ra 0$}.
\end{align}
as well as
\begin{align}\label{eq:bdd_init}
  \|u_{\eps 0} - u_0\|_{\leb1} \le 1
  \quad \text{and} \quad
  \|v_{\eps 0} - v_0\|_{\leb1} \le 1
  \qquad \text{for all $\eps \in (0, 1)$}.
\end{align}
(That this is indeed possible can rapidly be seen by a typical convolution argument.)
  
Moreover, for all $\eps \in (0, 1)$, we also fix a nonnegative function $f_\eps \in C_c^\infty([0, \infty))$
with
\begin{align*}
  f_\eps(s)
  \begin{cases}
    = s,    & 0 \le s \le \frac1\eps, \\
    \le s,  & \frac1\eps \lt s \lt  \frac2\eps, \\
    = 0,    & \frac2\eps \le s.
  \end{cases}
\end{align*}
With these preparations at hand, we can construct global solutions to certain approximate problems:

\begin{lem}\label{lem2.1}
Let $\varepsilon \in(0,1)$. Then there exist nonnegative functions
\begin{equation*}
\left\{\begin{array}{l}{u_{\varepsilon} \in C^{0}(\Ombar \times[0, \infty)) \cap C^{2,1}(\Ombar \times(0, \infty))} 
\\ {v_{\varepsilon} \in \bigcap_{q>n} C^{0}\left([0, \infty) ; W^{1, q}(\Omega)\right) \cap C^{2,1}(\Ombar \times(0, \infty))}\end{array}\right.
\end{equation*}
such that $\left(u_{\varepsilon}, v_{\varepsilon}\right)$ is a global classical solution of
\begin{equation}\label{e210}
\begin{cases}
\begin{array}{lll}
\medskip
u_{\varepsilon t}=\Delta u_{\varepsilon}-\nabla\cdot(f_{\varepsilon}(u_{\varepsilon})\nabla v_{\varepsilon})+\lambda(x) u_{\varepsilon}-\mu(x) u_{\varepsilon}^\kappa, &x\in \Omega,\ t>0,\\
\medskip
v_{\varepsilon t}=\Delta v_{\varepsilon}-v_{\varepsilon}+u_{\varepsilon}, &x\in \Omega,\ t>0,\\
\medskip
\partial_\nu \ue = 0,\ \partial_\nu \ve = 0, &x\in \partial \Omega,\ t>0,\\
\medskip
u_{\varepsilon}(x,0)=u_{\eps 0}(x),\ v_{\varepsilon}(x,0)=v_{\eps 0}(x), &x\in \Omega.
\end{array}
\end{cases}
\end{equation}
\end{lem}
\begin{proof}
  Existence and uniqueness in $\Ombar \times (0, \tmaxe)$ for some $\tmaxe \in (0, \infty]$
  can be shown as in~\cite[Theorem~3.1]{HorstmannWinklerBoundednessVsBlowup2005},
  nonnegativity follows from the maximum principle
  and as
  $\ol u_\eps(x, t) \defs \max\{\|u_{\eps 0}\|_{\leb\infty}, \frac2\eps\} \ure^{\|\lambda\|_{\leb\infty} t}$,
  $(x, t) \in \Ombar \times [0, \infty)$,
  defines a supersolution for the first equation in~\eqref{e210},
  the comparison principle asserts that $\limsup_{t \nea T} \|\ue(\cdot, t)\|_{\leb \infty} \le \ol u_\eps(\cdot, T)$
  is finite for all finite $T \in (0, \tmaxe]$,
  which in turn implies $\tmaxe = \infty$.
\end{proof}

Henceforth, for all $\eps \in (0, 1)$ we always denote the solution to~\eqref{e210} constructed in Lemma~\ref{lem2.1}
by $(\ue, \ve)$.

\section{A priori estimates}
In this section, we will first collect several a priori estimates.
Later on, in Lemma~\ref{lem2.6}, these will allow us to apply certain compactness theorems and then to construct a solution candidate for \eqref{e11}.
We begin by obtaining $L_{\loc}^\infty$-$L^1$ bounds both for $\ue$ and $\ve$ as well as a space-time bound for $\mu \ue^\kappa$.

\begin{lem}\label{lem2.2}
Let $T \gt 0$, $\eps \in (0, 1)$ and set $\lambda_1 \defs \|\lambda\|_{\leb\infty}$.
Then
\begin{equation}\label{e217}
\int_{\Omega} u_{\varepsilon}(\cdot, t) \leq \ure^{\lambda_1 T} \left( \int_{\Omega} u_{0} + 1 \right)
\end{equation}
as well as
\begin{equation}\label{e219}
\int_{\Omega} v_{\varepsilon}(\cdot, t)
\le \intom v_0 + 1 + \ure^{\lambda_1 T} \left( \intom u_0 + 1 \right)
\end{equation}
hold for all $t \in (0, T)$.
Moreover,
\begin{equation}\label{e218}
\int_0^T\int_\Omega \mu u_\varepsilon^\kappa
\leq \ure^{\lambda_1 T} \left( \intom u_0 + 1\right).
\end{equation}
\end{lem}
\begin{proof}
An integration of the first equation in $(\ref{e210})$ over $\Omega$ shows that
\begin{equation*}
\ddt\int_{\Omega}u_{\varepsilon} \le \lambda_1 \int_{\Omega} u_\varepsilon-\int_\Omega \mu u_\varepsilon^\kappa \qquad \text{in $(0, \infty)$}.
\end{equation*}
Thus, by an ODI comparison argument and the variations-of-constants formula,
\begin{align*}
      \intom \ue(\cdot, t)
  \le \ure^{\lambda_1 t} \intom u_{\eps 0}
      - \int_0^t  \ure^{\lambda_1 (t-s)} \intom \mu \ue^\kappa(x, s) \dx \ds
  \qquad \text{for all $t \in (0, T)$.}
\end{align*}
As $\ure^{\lambda_1 (t-s)} \ge 1$ whenever $t-s \ge 0$, we conclude
\begin{align*}
      \intom u(\cdot, t) + \int_0^t \intom \mu \ue^\kappa(x, s) \dx \ds
  \le \ure^{\lambda_1 t} \intom u_{\eps 0}
  \qquad \text{for all $t \in (0, T)$},
\end{align*}
which in view of \eqref{eq:bdd_init} immediately implies~\eqref{e217} and \eqref{e218}.

Moreover, integrating the second PDE in $(\ref{e210})$ results in
\begin{align*}
  \ddt \intom \ve + \intom \ve = \intom \ue \qquad \text{in $(0, T)$},
\end{align*}
so that another application of the variation-of-constants formula gives
\begin{align*}
      \intom \ve(\cdot, t)
  =   \ure^{-t} \intom v_{\eps 0} + \int_0^t \intom \ure^{-(t-s)} \ue(\cdot, s) \ds
  \le \intom v_{\eps 0} + \sup_{s \in (0, t)} \intom \ue(\cdot, s)
  \qquad \text{for all $t \in (0, T)$},
\end{align*}
which in virtue of \eqref{eq:bdd_init} and \eqref{e217} results in~\eqref{e219}.
\end{proof}

Next, we turn \eqref{e218} into a $L_{\loc}^\kappa$-$L^p$ bound for $\ue$,
making use of \eqref{cond:mu}, that is, the fact that $\intom \mu^{-s} \lt \infty$.
\begin{lem}\label{lm:u_l_kappa_l_p}
  Let $T \gt 0$
  and set $p \defs \frac{\kappa  s}{ s+1} \gt 0$.
  Then there is $C \gt 0$ such that
  \begin{align}\label{eq:u_l_kappa_l_p:est}
    \int_0^T \left( \intom \ue^p \right)^\frac{\kappa}{p} \lt C
    \qquad \text{for all $\eps \in (0, 1)$}.
  \end{align}
\end{lem}
\begin{proof}
  Since $\frac{p}{\kappa-p} = \frac{\frac{\kappa  s}{ s+1}}{\kappa - \frac{\kappa  s}{ s+1}} =  s$,
  by \eqref{cond:mu} there is $c_1 \gt 0$ such that $\intom \mu^{-\frac{p}{\kappa-p}} \lt c_1$.
  Applying Hölder's inequality, we then obtain
  \begin{align*}
        \intom \ue^p
   &=   \intom  \mu^{-\frac p\kappa} \cdot \mu^{\frac p\kappa} \ue^p
    \le \left( \intom \mu^{-\frac{p}{\kappa-p}} \right)^\frac{\kappa-p}{\kappa} \left( \intom \mu \ue^\kappa \right)^\frac{p}{\kappa}
    \le c_2 \left( \intom \mu \ue^\kappa \right)^\frac{p}{\kappa}
    \qquad \text{in $(0, T)$ for all $\eps \in (0, 1)$},
  \end{align*}
  where $c_2 \defs c_1^\frac{\kappa-p}{\kappa} \gt 0$.
  Therefore, upon integrating,
  \begin{align*}
        \int_0^T \left( \intom \ue^p \right)^\frac{\kappa}{p}
    \le c_2 \int_0^T \intom \mu \ue^\kappa
    \qquad \text{for all $\eps \in (0, 1)$}.
  \end{align*}
  The statement follows by~\eqref{e218}.
\end{proof}

As already discussed in the introduction,
for obtaining an $L^r$ bound for $\ve$, $\eps \in (0, 1)$, uniform both in $\eps$ and locally in time,
we can make use either of the $L_{\loc}^\infty$-$L^1$ bound~\eqref{e217} or the $L_{\loc}^\kappa$-$L^p$ bound \eqref{eq:u_l_kappa_l_p:est}.
Both these cases will be handles simultaneously in the following
\begin{lem}\label{lemv}
  Let $T>0$ and suppose~\eqref{eq:u_l_kappa_l_p:est} holds for some $p \ge 1$.
  For any
  \begin{align}\label{eq:v:r_cond}
    r \in \left[1, \max\left\{\frac{\kappa n}{[\frac{\kappa n}{p} - 2(\kappa-1)]_+}, \frac{n}{n-2} \right\} \right),
  \end{align}
  we can find $C \gt 0$ such that
  \begin{equation*}
    \left\|v_{\varepsilon}(\cdot, t)\right\|_{L^r(\Omega)} \leq C \qquad \text { for all } t \in(0, T) \text { and } \eps \in(0,1).
  \end{equation*}
\end{lem}
\begin{proof}
  Fix $\theta \in \{0, 1\}$.
  Since by~\eqref{e217} and assumption
  \begin{align*}
    \sup_{\eps \in (0, 1)} \|\ue\|_{L^\infty((0, T); \leb1)} \lt \infty
    \quad \text{and} \quad
    \sup_{\eps \in (0, 1)} \|\ue\|_{L^\kappa((0, T); \leb p)} \lt \infty,
  \end{align*}
  we infer that
  \begin{align}\label{eq:v_leb_r:ue_bdd}
    (\ue)_{\eps \in (0, 1)}
    \quad \text{is bounded in} \quad
    L^{p_{\theta,1}}\left((0, T); \leb{p_{\theta,2}}\right),
  \end{align}
  where $p_{\theta, 1} \defs \frac{\kappa}{1-\theta}$ and $p_{\theta, 2} \defs \frac{p}{1+(p-1)\theta}$.
  (As can be seen by Hölder's inequality, \eqref{eq:v_leb_r:ue_bdd} even holds for all $\theta \in [0, 1]$
  but in the sequel we will only make use of \eqref{eq:v_leb_r:ue_bdd} for $\theta \in \{0, 1\}$.)

  Let
  \begin{align}\label{eq:v:def_r_theta}
          1
    \le   r \lt r_\theta
    \defs \frac{\kappa n p}{[\kappa n - 2 (\kappa-1)p - (\kappa n - (\kappa n - 2)p ) \theta]_+}
    = \begin{cases}
      \frac{\kappa n}{[\frac{\kappa n}{p} - 2(\kappa-1)]_+}, & \theta = 0, \\
      \frac{n}{n-2}, & \theta = 1.
    \end{cases}
  \end{align}
  Since $r_\theta \gt p_{\theta, 2}$, we may without loss of generality assume $r \gt p_{\theta, 2}$ due to Hölder's inequality.
  We now make use of the variation-of-constants formula,
  well-known semigroup estimates (cf.\ \cite[Lemma~1.3~(i)]{WinklerAggregationVsGlobal2010}),
  Hölder's inequality and \eqref{eq:conv_init}
  to obtain that with $p_{\theta, 1}' \defs 1 - \frac1{p_{\theta, 1}}$,
  \begin{align*}
          \|\ve(\cdot, t)\|_{\leb r}
    &\le  \|\ure^{t(\Delta-1)} v_{\eps 0}\|_{\leb r}
          + \int_0^t \|\ure^{(t-s)(\Delta-1) t} \ue(\cdot, s)\|_{\leb r} \ds \\
    &\le  c_1 \|v_{\eps 0}\|_{\leb r}
          + c_1 \int_0^t \left(1 + (t-s))^{-\frac n2 (\frac1{p_{\theta, 2}} - \frac1r)}\right)
            \|\ue(\cdot, s)\|_{\leb{p_{\theta,2}}} \ds \\
    &\le  c_2 \|v_0\|_{\leb r}
          + c_1 \left( \int_0^T \left(1 + s^{-\frac n2 (\frac1{p_{\theta, 2}} - \frac1r)} \right)^{p_{\theta, 1}'} \right)^\frac1{p_{\theta_1}'}
          \|\ue\|_{L^{p_{\theta,1}}((0, T); \leb{p_{\theta,2}})}
  \end{align*}
  holds for all $\eps \in (0, 1)$, $t \in (0, T)$ and some $c_1, c_2 \gt 0$.
  As
  \begin{align*}
          -\frac n2 \left(\frac1{p_{\theta, 2}} - \frac1r \right) \left(1 - \frac1{p_{\theta,1}} \right)
    &\gt  -\frac n2
            \left(\frac{1+(p-1)\theta}{p} - \frac{\kappa n - 2 (\kappa-1)p - [\kappa n - (\kappa n - 2)p ] \theta}{\kappa n p} \right)
            \frac{\kappa}{\kappa-(1-\theta)} \\
    & =   -\frac n2
            \cdot \frac{\kappa n(p-1)\theta + 2 (\kappa-1)p + [\kappa n - (\kappa n - 2)p ] \theta}{np}
            \cdot \frac{1}{\kappa + \theta - 1} \\
    & =   - \frac{\kappa-1 + \theta}{\kappa + \theta - 1}
    =     - 1
  \end{align*}
  and because of \eqref{eq:v_leb_r:ue_bdd},
  the right hand side therein is bounded independently of $\eps \in (0, 1)$ and $t \in (0, T)$.
  Finally, the statement is a direct consequence of \eqref{eq:v:def_r_theta}.
\end{proof}

Before using the information gathered in Lemma~\ref{lm:u_l_kappa_l_p} and Lemma~\ref{lemv} to obtain further a priori estimates,
we show that several parameters can be chosen suitably.
The following lemma is made possible precisely due to the conditions~\eqref{cond:kappa} and \eqref{cond:mu}.
Its importance will become apparent in the proof of Lemma~\ref{lem2.3} below.

\begin{lem}\label{lem:params}
  Set $p \defs \frac{\kappa s}{s+1}$, $p' \defs \frac{p}{p-1}$ and $\kappa' \defs \frac{\kappa}{\kappa-1}$.
  Then there exist $q, r \gt 1$ satisfying $q \gt p'$, \eqref{eq:v:r_cond},
  \begin{align}\label{eq:params:statement}
    \theta(q, r) \defs \frac{\frac{1}{r}-\frac{1}{q}}{\frac{1}{n}-\frac{1}{2}+\frac{1}{r}} \in (0, 1)
    \quad \text{and} \quad
    \kappa' \theta(q, r) \lt 2.
  \end{align}
\end{lem}
\begin{proof}
  Noting that $0 \lt r \lt q \lt \frac{2n}{n-2}$ implies $\theta(q, r) \in (0, 1)$
  and making first use of that $\kappa \gt \frac{2n(s+1)}{(n+2)s}$ by \eqref{cond:kappa} in calculating
  \begin{align*}
        p'
    =   \frac{1}{1-\frac1p}
    =   \frac{1}{1-\frac{s+1}{\kappa s}}
    \lt \frac{1}{1-\frac{(n+2)(s+1)}{2n (s+1)}}
    =   \frac{2n}{2n - (n+2)}
    =   \frac{2n}{n-2},
  \end{align*}
  we see that the statement follows
  once we show that there are $q \in (p', \frac{2n}{n-2})$ and $r \in (1, q)$
  satisfying~\eqref{eq:v:r_cond} and $\kappa' \theta(q, r) \lt 2$.

  If $n=2$, this can easily be achieved by choosing $q \in (p', \frac{2n}{n-2})$ arbitrarily 
  and $r \ge 1$ close enough to $q$ such that $\theta(q, r) = 1-\frac rq \lt \frac2{\kappa'}$,
  as \eqref{eq:v:r_cond} is then equivalent to $r \in [1, \infty)$.
  Moreover, for any $n \gt 2$ and $\kappa \ge 2$,
  we may choose $q \in (p', \frac{2n}{n-2})$ arbitrarily and $r \in (1, q)$ such that \eqref{eq:v:r_cond} is fulfilled,
  as then $\kappa' \theta(q, r) \le 2 \theta(q, r) \lt 2$.

  Thus, regarding the remaining case $n \ge 3$ and $\kappa \le 2$,
  we will now make use of the yet unused condition in~\eqref{cond:kappa},
  namely of $\kappa \gt \min\left\{ \frac{2(n-1) s + n}{n s}, \frac{2(n+2) s + 2n}{(n+4) s} \right\}$,
  and divide the remainder of this proof in two parts.

  \emph{Case 1:} $n \gt 2$ and $\kappa \in (\frac{2(n-1)s+n}{ns}, 2)$.
    We set $\qsup \defs \frac{\kappa n}{n-2}$
    and calculate
    \begin{align*}
          \frac{\qsup}{p'}
      =   \frac{\kappa n}{n-2} \left(1 - \frac{s+1}{\kappa s} \right)
      =   \frac{\kappa n - n - \frac ns}{n-2}
      \gt \frac{2(n-1) + \frac ns - n - \frac ns}{n-2}
      =   1,
    \end{align*}
    hence we may fix $q \in (p', \min\{\frac{2n}{n-2}, \qsup\})$.
    
    Wet set $\rsup \defs \frac{n}{n-2}$ and note that $r \in (1, \rsup)$ implies \eqref{eq:v:r_cond}.
    If $q \le \rsup$ and hence $\theta(q, \rsup) \le 0$,
    we may choose $r \in (1, q) = (1, \min\{q, \rsup\})$ sufficiently large such that $\kappa' \theta(q, r) \lt 2$.
    Thus, we may assume $q \gt \rsup$.
    We can then fix $r \in (1, \rsup) = (1, \min\{q, \rsup\})$ such that $\kappa' \theta(q, r) \lt 2$
    since due to $\frac{\kappa'}{\kappa} = \kappa'-1$ we may calculate
    \begin{align*}
          \kappa' \theta(q, r)
      \lt \kappa' \cdot \frac{\frac1{\rsup}-\frac1{\qsup}}{\frac{2-n}{2n}+\frac{1}{\rsup}}
      =   \kappa' \cdot \frac{\frac{n-2}{n} - \frac{n-2}{\kappa n}}{\frac{2-n}{2n} + \frac{n-2}{n}}
      =   \kappa' \cdot \frac{(1-\frac1\kappa) \frac{n-2}{n}}{\frac{n-2}{2n}}
      =   2.
    \end{align*}

  \emph{Case 2:} $n \gt 2$ and $\kappa \in (\frac{2(n+2)s+2n}{(n+4)s}, 2)$.
    We set $\qsup \defs \frac{\kappa^2 ns}{[(\kappa s - (s+1)) (\kappa n - (n+4)) + n(s+1) - 4]_+}$.
    As $\kappa \gt \frac{2n(s+1)}{(n+2)s} \gt \frac{s+1}{s}$ ensures that the denominator in the first term in the following calculation is positive,
    we may estimate
    \begin{align*}
            \frac{n(s+1) - 4}{[\kappa s - (s+1)] (n + 4)}
      &\lt  \frac{n(s+1) - 4}{2(n+2)s + 2n - (n+4)(s+1)}
       =    \frac{ns + n - 4}{ns + n - 4}
       =    1
    \end{align*}
    and because of $p' = \frac{\kappa s}{\kappa s - (s+1)}$,
    we obtain
    \begin{align*}
            \frac{p'}{\qsup}
      \le   \frac{\kappa n - (n+4) + \frac{n(s+1) - 4}{\kappa s - (s+1)}}{\kappa n}
      \lt   \frac{\kappa n}{\kappa n}
      =     1,
    \end{align*}
    allowing us to again fix $q \in (p', \min\{\frac{2n}{n-2}, \qsup\})$.

    Making use of $p \le \kappa \lt 2$,
    we moreover see that
    $\rsup \defs \frac{\kappa n}{\frac{\kappa n}{p} - 2(\kappa-1)} \le \frac{\kappa n}{n - 2(\kappa-1)} \lt \frac{2n}{n-2}$
    and, arguing as in Case~1, we may assume $q \gt \rsup$.
    Rewriting  $\qsup$ and $\rsup$ as
    $\qsup = \frac{\kappa^2 ns}{[\kappa^2 n s - \kappa n(2s+1) + 2(n(s+1) - 2(\kappa-1)s)_+}$ and
    $\rsup = \frac{\kappa n s}{n(s+1) - 2(\kappa-1)s}$, respectively, 
    and 
    this time not only making use of $\frac{\kappa'}{\kappa} = \kappa' - 1$,
    but also of $(\kappa' - 1) \kappa = \kappa'$ and $\kappa' (\kappa - 1) = \kappa$,
    we obtain
    \begin{align*}
            \kappa' \cdot \frac{\frac{1}{\rsup}-\frac{1}{q}}{\frac{1}{n}-\frac{1}{2}+\frac{1}{\rsup}}
      &\lt  \frac{\frac{\kappa'}{\rsup}-\frac{\kappa'}{\qsup}}{\frac{1}{n}-\frac{1}{2}+\frac{1}{\rsup}} \\
      &\le  \frac{2 \kappa' [n(s+1) - 2(\kappa-1)s] - 2 (\kappa'-1) \big[ \kappa^2 n s - \kappa n(2s+1) + 2[n(s+1) - 2(\kappa-1)s] \big]}
              {2\kappa s - \kappa n s + 2 [n(s+1) - 2(\kappa-1)s]} \\
      &=    \frac{2(\kappa'-1) \kappa [n (2s+1) - \kappa ns] - (2 \kappa' - 4) [n(s+1) - 2(\kappa-1)s] }
              {2\kappa s - \kappa n s + 2 [n(s+1) - 2(\kappa-1)s]} \\
      &=    \frac{2\kappa'(\kappa-1) (2s - ns) + 4 \big[n(s+1) - 2(\kappa-1)s \big]}
              {2\kappa s - \kappa n s + 2 [n(s+1) - 2(\kappa-1)s]}
      =     2.
    \end{align*}
    Therefore, we may again fix $r \in (1, \rsup) = (1, \min\{q, \rsup\})$ such that $\kappa' \theta(q, r) \lt 2$.
    By the definition of $\rsup$, we finally see that $r$ satisfies \eqref{eq:v:r_cond}.
\end{proof}

We now further gain certain space-time bounds, inter alia for $\ue \ve$ and $\nabla \ve$.
This is achieved by testing the second equation in \eqref{e11} with $\ve$
and making use of the previous lemma in conjunction with the Gagliardo--Nirenberg inequality
in order to handle the production term $+\ue$ in that equation.

\begin{lem}\label{lem2.3}
There exist $p \gt 1$, $q \gt \frac{p}{p-1}$ and $\gamma \gt \frac{\kappa}{\kappa-1}$
such that for all $T \gt 0$
there is $C \gt 0$ with the property that
for all $\eps \in (0, 1)$ we have
\begin{equation}\label{eq:ddt_v2:est}
    \int_0^T \left(\int_{\Omega}u_{\varepsilon}^{p}\right)^{\frac{\kappa}{p}}
  + \int_0^T \left(\int_{\Omega}v_{\varepsilon}^{q}\right)^{\frac{\gamma}{q}}
  + \int_0^T \int_{\Omega}\left|\nabla v_{\varepsilon}\right|^{2}
  \leq C.
\end{equation}
\end{lem}
\begin{proof} 
We multiply the second equation in $\eqref{e210}$ with $v_\varepsilon$ and integrate by parts to
find that
\begin{equation}\label{e221}
\frac12\ddt\int_{\Omega}v_{\varepsilon}^2=-\int_{\Omega}|\nabla v_\varepsilon|^2
-\int_\Omega v_\varepsilon^2+\int_\Omega u_\varepsilon v_\varepsilon
\qquad \text{in $(0, T)$ for all $\eps \in (0, 1)$}.
\end{equation}

We set again $p \defs \frac{\kappa s}{s+1}$, $p' \defs \frac{p}{p-1}$ and $\kappa' \defs \frac{\kappa}{\kappa-1}$.
By Hölder's and Young's inequalities, we then have
\begin{equation}\label{e222}
      \intom \ue \ve
  \le \left( \intom \ue^p \right)^\frac1p \left( \intom \ve^{p'} \right)^\frac{1}{p'}
  \le \left( \intom \ue^p \right)^\frac{\kappa}{p} + \left( \intom \ve^{p'} \right)^\frac{\kappa'}{p'}
  \qquad \text{in $(0, T)$ for all $\eps \in (0, 1)$}.
\end{equation}

Lemma~\ref{lem:params} allows us to fix $q \gt p'$ and $r \ge 1$ satisfying \eqref{eq:v:r_cond} and \eqref{eq:params:statement}.
In particular, $\theta \defs \theta(q, r) \lt \frac{2}{\kappa'}$,
hence there is $\gamma \gt \kappa'$ such that still $\theta \gamma \le 2$.
Thus, we may make use of the Gagliardo--Nirenberg inequality (applicable due to the first condition in~\eqref{eq:params:statement}),
Lemma~\ref{lemv} (which we may employ because of Lemma~\ref{lm:u_l_kappa_l_p} and since \eqref{eq:v:r_cond} holds)
and Young's inequality
to obtain $c_1, c_2, c_3 \gt 0$ such that
\begin{equation}\label{e2222a}
\begin{split}
        \left\|\ve(\cdot, t)\right\|_{L^{q}(\Omega)}^{\gamma}
  &\leq c_1\|\nabla \ve(\cdot, t)\|_{L^{2}(\Omega)}^{\theta\gamma}
        \left\|\ve(\cdot, t)\right\|_{L^{r}(\Omega)}^{(1-\theta)\gamma}
        + c_1\left\|\ve(\cdot, t)\right\|_{L^{r}(\Omega)}^{\gamma} \\
  &\leq c_2 \left\|\nabla \ve(\cdot, t)\right\|_{L^{2}(\Omega)}^{\theta \gamma} + c_2 \\
  &\leq c_3 \left\|\nabla \ve(\cdot, t)\right\|_{L^{2}(\Omega)}^2 + c_3
  \qquad \text{for all $\eps \in (0, 1)$ and $t \in (0, T)$}.
\end{split}
\end{equation}

Here another application of Young's inequality provides us with $c_4 \gt 0$ such that
\begin{equation}\label{eq:ddt_v2:v}
        \left\|\ve(\cdot, t)\right\|_{L^{p'}(\Omega)}^{\kappa'}
   \leq \frac12 \left\|\nabla \ve(\cdot, t)\right\|_{L^{2}(\Omega)}^2 + c_4
  \qquad \text{for all $t \in (0, T)$ and $\eps \in (0, 1)$}.
\end{equation}

Since that the asserted bound for $\ue$ was already proven in \eqref{e222},
combining~\eqref{e221} with \eqref{eq:u_l_kappa_l_p:est} and \eqref{eq:ddt_v2:v},
after integrating results  (inter alia) in
\begin{align*}
  \int_0^T \intom |\nabla \ve|^2 \le c_5 
  \qquad \text{for all and $\eps \in (0, 1)$}
\end{align*}
for some $c_5 \gt 0$.
Together with~\eqref{e2222a} this implies \eqref{eq:ddt_v2:est}.
\end{proof}
The fact that we could derive an $L_{\loc}^2$-$L^2$ bound for $\nabla v$ in Lemma~\ref{lem2.3}
relied on Lemma~\ref{lem:params} and thus, by extension, also on our main condition \eqref{cond:kappa}.
Its importance first becomes apparent in the following lemma, providing a space-time bound for $\frac{|\nabla \ue|^2}{(\ue + 1)^2}$.
\begin{lem}\label{lem2.4}
For any $T \gt 0$ there is $C \gt 0$ such that
\begin{equation}\label{e223}
\int_{0}^{T} \int_{\Omega} \frac{\left|\nabla u_{\varepsilon}\right|^{2}}{\left(u_{\varepsilon}+1\right)^{2}} \leq C
\qquad \text{for all $\eps \in (0, 1)$}.
\end{equation}
\end{lem}
\begin{proof}
Testing the first equation in (\ref{e210}) with $\frac{1}{u_{\varepsilon}+1}$
and integrating by parts gives
 \begin{equation}\label{e2231}
\begin{aligned} 
&\pe -\ddt \int_{\Omega} \ln \left(u_{\varepsilon}+1\right) \\
&=-\int_{\Omega} \frac{1}{u_{\varepsilon}+1} \Delta u_{\varepsilon}+\int_{\Omega} \frac{1}{u_{\varepsilon}+1} \nabla \cdot\left(f_{\varepsilon}\left(u_{\varepsilon}\right) \nabla v_{\varepsilon}\right) -\int_\Omega\frac{\lambda u_\varepsilon}{u_\varepsilon+1}+\int_\Omega\frac{\mu u_\varepsilon^\kappa}{u_\varepsilon+1}\\ 
&\le-\int_{\Omega} \frac{\left|\nabla u_{\varepsilon}\right|^{2}}{\left(u_{\varepsilon}+1\right)^{2}}+\int_{\Omega} \frac{f_{\varepsilon}\left(u_{\varepsilon}\right)}{\left(u_{\varepsilon}+1\right)^{2}} \nabla u_{\varepsilon} \cdot \nabla v_{\varepsilon} + \lambda_1 |\Omega| + \intom \frac{\mu \ue^\kappa}{\ue+1}
\qquad \text{in $(0, T)$ for all $\eps \in (0, 1)$},
\end{aligned}
\end{equation}
where again $\lambda_1 \defs \|\lambda\|_{\leb\infty}$.
By using Young’s inequality, we see that here
\begin{equation*}
\begin{aligned}
\left|\int_{\Omega} \frac{f_{\varepsilon}\left(u_{\varepsilon}\right)}{\left(u_{\varepsilon}+1\right)^{2}} \nabla u_{\varepsilon} \cdot \nabla v_{\varepsilon}\right| & \leq \frac{1}{2} \int_{\Omega} \frac{\left|\nabla u_{\varepsilon}\right|^{2}}{\left(u_{\varepsilon}+1\right)^{2}}+\frac{1}{2} \int_{\Omega} \frac{f_{\varepsilon}^{2}\left(u_{\varepsilon}\right)}{\left(u_{\varepsilon}+1\right)^{2}}\left|\nabla v_{\varepsilon}\right|^{2} \\
& \leq \frac{1}{2} \int_{\Omega} \frac{\left|\nabla u_{\varepsilon}\right|^{2}}{\left(u_{\varepsilon}+1\right)^{2}}+\frac{1}{2} \int_{\Omega}\left|\nabla v_{\varepsilon}\right|^{2}
\qquad \text{in $(0, T)$ for all $\eps \in (0, 1)$}
\end{aligned}
\end{equation*}
because $f_\varepsilon(u_\varepsilon)\leq u_\varepsilon$ in $\Omega \times (0, T)$ for all $\eps \in (0, 1)$.
As $0 \leq \ln (s+1) \leq s$ for $s \ge 0$,
upon integrating in time we thus obtain from (\ref{e2231}) that
\begin{equation*}
\begin{aligned} 
\frac{1}{2} \int_{0}^{T} \int_{\Omega} \frac{\left|\nabla u_{\varepsilon}\right|^{2}}{\left(u_{\varepsilon}+1\right)^{2}} & \leq \frac{1}{2} \int_{0}^{T} \int_{\Omega}\left|\nabla v_{\varepsilon}\right|^{2}+\int_{\Omega} \ln \left(u_{\varepsilon}(\cdot, T)+1\right)-\int_{\Omega} \ln \left(u_{0}+1\right)+ \lambda_1 T |\Omega| + \int_0^T \intom \frac{\mu \ue^\kappa}{\ue+1} \\ 
& \leq \frac{1}{2} \int_{0}^{T} \int_{\Omega}\left|\nabla v_{\varepsilon}\right|^{2}+\int_{\Omega} u_{\varepsilon}(\cdot, T) + \lambda_1 T |\Omega| + \int_0^T \intom \frac{\mu \ue^\kappa}{\ue+1}
\qquad \text { for all } \eps \in (0, 1).
\end{aligned}
\end{equation*}
An application of Lemma~\ref{lem2.3} and Lemma~\ref{lem2.2} therefore yields (\ref{e223})
for an appropriately chosen $C \gt 0$, as desired.
\end{proof}
As a last a priori estimate in our collection of these,
we prove certain bounds of the time derivates of $\ln (\ue+1)$ and $\ve$.
In conjunction with already obtained bounds,
these will allow us to apply the Aubin--Lions lemma, asserting (inter alia) a.e.\ pointwise convergence of $\ue$ and $\ve$
along certain subsequences.
\begin{lem}\label{lem2.5}
For all $T \gt 0$ there exists $C>0$ such that for each $\varepsilon\in (0,1)$,
\begin{equation}\label{e224}
\int_{0}^{T}\|\partial_{t}\ln(u_{\varepsilon}(\cdot,t)+1)\|_{(W^{n,2}(\Omega))^{*}}\dt\leq C
\end{equation}
and
\begin{equation}\label{e225}
\int_{0}^{T}\|v_{\varepsilon t}(\cdot,t)\|_{(W^{n,2}(\Omega))^{*}}\dt\leq C.
\end{equation}
\end{lem}

\begin{proof}
For arbitrary $t>0$ and $\psi\in C^{\infty}(\Ombar)$, multiplying the first equation in $(\ref{e210})$ by $\frac{\psi}{u_{\varepsilon}(\cdot, t)+1}$ and using Young's inequality as well as the Cauchy--Schwarz inequality, we see that
\begin{align*}
  &\pe  {\left|\int_{\Omega} \partial_{t} \ln \left(u_{\varepsilon}(\cdot, t)+1\right) \cdot \psi\right|} \\
  &=    \left|
          \intom \frac{|\nabla \ue|^2 \psi}{(\ue+1)^2}
          - \intom \frac{\nabla \ue \cdot \nabla \psi}{\ue+1}
          - \intom \frac{f_\eps(\ue) \nabla \ue \cdot \nabla \ve \psi}{(\ue+1)^2}
          + \intom \frac{f_\eps(\ue) \nabla \ve \cdot \nabla \psi}{\ue+1}
          + \intom \frac{\lambda \ue \psi}{\ue+1}
          - \intom \frac{\mu \ue^\kappa \psi}{\ue+1}
        \right|\\
  &\le  \left\{
          2 \intom \frac{|\nabla \ue|^2}{(\ue+1)^2}
          + \intom |\nabla \ve|^2
          + 1
          + \lambda_1 |\Omega|
          + \intom \mu \ue^\kappa
        \right\}
        \cdot\Big\{\|\nabla \psi\|_{L^{2}(\Omega)}+\|\psi\|_{L^{\infty}(\Omega)} \Big\}
\end{align*}
for all $\eps \in (0, 1)$,
where $\lambda_1 \defs \|\lambda\|_{\leb\infty}$.

Noting that $\sob n2 \embed \sob12$ and $\sob n2 \embed \leb\infty$,
we can fix $c_1>0$ such that 
\begin{align*}
    \|\nabla \psi\|_{L^{2}(\Omega)}+\|\psi\|_{L^{\infty}(\Omega)} \le c_1 \|\psi\|_{\sob n2}
    \qquad \text{for all $\psi \in \sob n2$}
\end{align*}
and obtain
\begin{align*}
    {\left\|\partial_{t} \ln \left(u_{\varepsilon}(\cdot, t)+1\right)\right\|_{\left(W^{n,2}(\Omega)\right)^*}}
&\leq c_{1} \cdot\left\{2 \int_{\Omega} \frac{\left|\nabla u_{\varepsilon}\right|^{2}}{\left(u_{\varepsilon}+1\right)^{2}}+\int_{\Omega}\left|\nabla v_{\varepsilon}\right|^{2}+1+\lambda_1 |\Omega| +\int_\Omega \mu u_\varepsilon^\kappa\right\}
\end{align*}
for all $\eps \in (0, 1)$, $t \in (0, T)$ and $\psi \in \con\infty$.
Thus, after an integration in time and as $\con\infty$ is dense in $\sob n2$,
from Lemma~\ref{lem2.4}, Lemma~\ref{lem2.3} and Lemma~\ref{lem2.2}, we infer~\eqref{e224}.

Similarly, by testing the second equation in $(\ref{e210})$ with $\psi \in \con\infty$ and using H\"older's inequality, we have
\begin{equation*}
\begin{split}
\left|\int_{\Omega}v_{\varepsilon t}(\cdot,t)\psi dx\right|&=\left|-\int_{\Omega}\nabla v_\varepsilon\cdot\nabla\psi-\int_{\Omega}v_\varepsilon\psi+\int_{\Omega}u_\varepsilon\psi\right|\\
&\leq\left\{\left\{\int_\Omega|\nabla v_\varepsilon|^{2}\right\}^{\frac{1}{2}}+\int_{\Omega}v_\varepsilon+\int_{\Omega}u_\varepsilon\right\}\cdot \Big\{\|\nabla \psi\|_{L^2(\Omega)}+\|\psi\|_{L^{\infty}(\Omega)}\Big\}\\
&\leq\left\{\int_\Omega|\nabla v_\varepsilon|^{2}+\frac{1}{4}+\int_{\Omega}v_\varepsilon+\int_{\Omega}u_\varepsilon\right\}\cdot c_1\|\psi\|_{W^{n,2}(\Omega)}
\qquad \text{for all $\eps \in (0, 1)$ and $t \in (0, T)$}.
\end{split}
\end{equation*}
As a consequence,
\begin{equation*}
\|v_{\varepsilon t}(\cdot,t)\|_{(W^{n,2}(\Omega))^*}\leq c_1\left\{\int_\Omega|\nabla v_\varepsilon|^{2}+\frac14+\int_{\Omega}v_\varepsilon+\int_{\Omega}u_\varepsilon\right\}
\qquad \text{for all $\eps \in (0, 1)$ and $t \in (0, T)$},
\end{equation*}
so that $(\ref{e225})$ results from Lemma~\ref{lem2.2} and Lemma~\ref{lem2.3}.
\end{proof}
We now combine the a priori estimates gained in this section with well-known compactness theorems
in order to obtain a solution candidate $(u, v)$ as the limit of $(\ue, \ve)$ in certain topologies.
At the end of the succeeding section, we will then prove that a pair $(u, v)$ constructed in this way
is indeed a global generalized solution of \eqref{e11}.
\begin{lem}\label{lem2.6}
There exist nonnegative functions $u$ and $v$ defined on $\Omega\times (0,\infty)$,
$\tilde p, \tilde q, \tilde \kappa, \tilde \gamma \gt 1$
with
\begin{align}\label{eq:conv:tilde}
  \tilde q = \frac{\tilde p}{\tilde p - 1}
  \quad \text{and} \quad
  \tilde \gamma = \frac{\tilde \kappa}{\tilde \kappa - 1}
\end{align}
as well as a sequence $(\varepsilon_{k})_{k\in\mathbb{N}}\subset(0,1)$ such that $\varepsilon_{k}\searrow 0$ as $k\to\infty$ and 
\begin{alignat}{2}
u_{\varepsilon} &\to u                                                     \qquad&& \text{in } L_{\loc}^{\tilde \kappa}([0, \infty); \leb {\tilde p}) \text{ and a.e.\ in } \Omega\times(0,\infty),\label{e228}\\
\ue(\cdot, t) &\to u(\cdot, t)                                             \qquad&& \text{in $\leb1$ for a.e.\ $t \gt 0$} \label{e229}\\
\ln \left(u_{\varepsilon}+1\right) &\rightharpoonup \ln (u+1)              \qquad&& \text{in }L_{\loc}^2([0,\infty);W^{1,2}(\Omega)),\label{e230}\\
\mu^\frac1\kappa \ue &\rightharpoonup \mu^\frac1\kappa u                   \qquad&& \text{in } L_{\loc}^\kappa(\ol{\Omega}\times [0,\infty)),\label{e234}\\
v_{\varepsilon} &\to v                                                     \qquad&& \text{in } L_{\loc}^{\tilde \gamma}([0, \infty); \leb {\tilde q} \text{ and a.e.\ in } \Omega\times(0,\infty),\label{e231}\\
\nabla v_{\varepsilon}&\rightharpoonup \nabla v                            \qquad&& \text{in } L_{\loc}^2(\ol{\Omega}\times[0,\infty)),\label{e232}\\
\ue \ve &\ra uv                                                            \qquad&& \text{in } L_{\loc}^1(\Ombar \times [0, \infty)).\label{eq:conv:uv}
\end{alignat}
as $\varepsilon=\varepsilon_{k}\searrow 0$.
Moreover, the pair $(u,v)$ has the properties \eqref{e21}--\eqref{e23} in Definition~\ref{defi21}.
\end{lem}
\begin{proof}
  From Lemma~\ref{lem2.3}, Lemma~\ref{lem2.4} and Lemma~\ref{lem2.5},
  a sequence $(\eps_k)_{k \in \N}$ with $\eps_k \ra 0$
  and $\ln(u_{\eps_k} +1) \ra z$ as well as $v_{\eps_k} \ra v$ in $L_{\loc}^2(\Ombar \times [0, \infty))$ as $k \ra \infty$
  for certain $z, v \in L_{\loc}^2(\Ombar \times [0, \infty))$
  can be obtained through direct applications of the Aubin--Lions lemma (combined with a diagonalization argument).
  Along a further subsequence, which we also denote by $(\eps_k)_{k \in \N}$ for convenience,
  we also have $\ln(u_{\eps_k} +1) \ra z$ a.e.\ (in $\Omega \times (0, \infty)$) as well as $v_{\eps_k} \ra v$ a.e.\ as $k \ra \infty$,
  which in turn implies $z, v \ge 0$ a.e.\ due to Lemma~\ref{lem2.1}.
  Moreover, as $\ln(\cdot+1)$ is a homeomorphism on $[0, \infty)$, 
  we conclude $u_{\eps_k} \ra \ure^{z} - 1 \sfed u \ge 0$ a.e.\ as $k \ra \infty$, thus \eqref{e22} holds.

  Let $p, q, \gamma \gt 1$ be as in Lemma~\ref{lem2.3}.
  In particular, $q \gt \frac{p}{p-1}$ and $\gamma \gt \frac{\kappa}{\kappa-1}$,
  hence we may fix $\tilde p \in (1, p), \tilde q \in (1, q), \tilde \gamma \in (1, \gamma), \tilde \kappa \in (1, \kappa)$
  such that \eqref{eq:conv:tilde} holds.
  Then Vitali's theorem and \eqref{eq:ddt_v2:est} assert \eqref{e228} and \eqref{e231},
  while \eqref{eq:conv:uv} is a direct consequence thereof and of Hölder's inequality.
  Upon extracting a further subsequence, if necessary, we also obtain \eqref{e229} from \eqref{e228}.

  As pointwise convergence (almost everywhere) has already been shown,
  the weak convergences in~\eqref{e230}, \eqref{e234} and \eqref{e232}
  are consequences of the estimates \eqref{e223}, \eqref{e218} and \eqref{eq:ddt_v2:est}, respectively.

  Finally, \eqref{e21} follows from \eqref{e228}, \eqref{e231} and \eqref{e232},
  while \eqref{e230} implies \eqref{e23}.
\end{proof}

\section{Proof of Theorem~\ref{main_result}}\label{sec:proof}
The goal of this section is to prove Theorem~\ref{main_result};
that is, we will show that the pair $(u, v)$ constructed in Lemma~\ref{lem2.6} is a generalized solution of \eqref{e11}.
While Lemma~\ref{lem2.6} already asserts that all integrals in \eqref{e24}, \eqref{e25} and \eqref{e26} exist,
the convergence statements in Lemma~\ref{lem2.6} are yet insufficient
to obtain
\begin{align*}
-\int_{0}^{\infty} \int_{\Omega} \frac{f_{\varepsilon}\left(u_{\varepsilon}\right)}{u_{\varepsilon}+1}\left(\nabla \ln \left(u_{\varepsilon}+1\right) \cdot \nabla v_{\varepsilon}\right) \varphi
  \ra -\int_0^\infty \intom \frac{u}{u+1} (\nabla \ln(u+1) \cdot \nabla v) \varphi,
\end{align*}
along some sequence null sequence $(\eps_k)_{k \in \N}$.
Thus, as a last preparation, we now proceed to obtain strong convergence of
$(\nabla \ve)_{\eps \in (0, 1)}$ in $L_{\loc}^2(\Ombar \times [0, \infty))$ along a suitable sequence.

\begin{lem}\label{lem2.9}
Let $v$ and $(\varepsilon_{k})_{k\in\mathbb{N}}$ be as in Lemma~\ref{lem2.6}. Then, for each $T>0$, we have
\begin{equation}\label{nablavl2}
\nabla v_{\varepsilon}\to\nabla v \qquad \text{in $L^{2}(\Omega\times(0,T))$  as $\varepsilon=\varepsilon_{k}\searrow0$}.
\end{equation}
\end{lem}
\begin{proof}
We follow \cite[Lemma~4.5]{WinklerRoleSuperlinearDamping2019}.
For $\eps, \eps^\prime \in (0, 1)$, Lemma~\ref{lem2.1} asserts that $\ve - \vep$ solves
\begin{align*}
  (\ve - \vep)_t = \Delta (\ve - \vep) - (\ve - \vep) + (\ue + \uep)
  \qquad \text{in $\Omega \times (0, \infty)$},
\end{align*}
so that testing with $\ve - \vep$ leads to
\begin{equation*}
\begin{aligned} 
&\pe\frac{1}{2} \int_{\Omega}\left(v_{\varepsilon}(\cdot, T)-v_{\varepsilon^{\prime}}(\cdot, T)\right)^{2}+\int_{0}^{T} \int_{\Omega}\left|\nabla v_{\varepsilon}-\nabla v_{\varepsilon^{\prime}}\right|^{2}+\int_{0}^{T} \int_{\Omega}\left(v_{\varepsilon}-v_{\varepsilon^{\prime}}\right)^{2} & \\
&= \frac{1}{2} \int_{\Omega}\left(v_{0 \varepsilon}-v_{0 \varepsilon^{\prime}}\right)^{2}+\int_{0}^{T} \int_{\Omega}\left(u_{\varepsilon}-u_{\varepsilon^{\prime}}\right)\left(v_{\varepsilon}-v_{\varepsilon^{\prime}}\right)
\qquad \text{for all $\eps, \eps^\prime \in (0, 1)$}.
\end{aligned}
\end{equation*}
In view of \eqref{e232}, \eqref{eq:conv_init}, \eqref{e231} and \eqref{eq:conv:uv},
we hence infer by using H\"older's inequality,
\begin{equation*}
\begin{aligned} 
&\pe\int_{0}^{T} \int_{\Omega}\left|\nabla v_{\varepsilon}-\nabla v\right|^{2}\\
& \leq \liminf _{j \ra \infty} \int_{\Omega}\left|\nabla v_{\varepsilon}-\nabla v_{\eps_j}\right|^{2} \\
& \leq \liminf _{j \ra \infty}\left\{\frac{1}{2} \int_{\Omega}\left(v_{0 \varepsilon}-v_{0 \eps_j}\right)^{2}+\int_{0}^{T} \int_{\Omega}\left(u_{\varepsilon}-u_{\varepsilon_j}\right)\left(v_{\varepsilon}-v_{\varepsilon_j}\right)\right\} \\
& =\frac{1}{2} \int_{\Omega}\left(v_{0 \varepsilon}-v_{0}\right)^{2}+\int_{0}^{T} \int_{\Omega}\left(u_{\varepsilon}-u\right)\left(v_{\varepsilon}-v\right)\\
&\leq \frac{1}{2} \int_{0}\left(v_{0 \varepsilon}-v_{0}\right)^{2}
+\left\|u_{\varepsilon}-u\right\|_{L^{\tilde \kappa}((0, T); \leb{\tilde p})}
 \left\|v_{\varepsilon}-v\right\|_{L^{\tilde \gamma}((0, T); \leb{\tilde q})}
 \qquad \text{for all $\eps \in (0, 1)$},
\end{aligned}
\end{equation*}
where $\tilde p, \tilde q, \tilde \kappa$ and $\tilde \gamma$ are as in Lemma~\ref{lem2.6}.
Thus, \eqref{nablavl2} is a consequence of \eqref{eq:conv_init}, \eqref{e228} and \eqref{e231}.
\end{proof}

With the above lemma at hand, we can now prove Theorem~\ref{main_result}.

\begin{proof}[Proof of Theorem~\ref{main_result}]
Let $(u, v)$ and $(\eps_k)_{k \in \N}$ be as in Lemma~\ref{lem2.6}.
To prove that $(u,v)$ is a global generalized solution in the sense of Definition~\ref{defi21}, it is sufficient to show that (\ref{e24}), (\ref{e25}) and (\ref{e26}) hold,
since validity of \eqref{e21}--\eqref{e23} has already been asserted in Lemma~\ref{lem2.6}.

By integrating the first equation in (\ref{e210}) both in space and in time, we see that
\begin{equation}\label{e42}
\int_{\Omega} u_{\varepsilon}(\cdot, T)-\int_{\Omega} u_{\eps 0}=\int_{0}^{T} \int_{\Omega} \lambda u_{\varepsilon}-\int_{0}^{T} \int_{\Omega} \mu u_{\varepsilon}^\kappa
\qquad \text{for all $\eps \in (0, 1)$ and $T \gt 0$}.
\end{equation}
As due to (\ref{e234}) we have 
\begin{equation*}
\int_{0}^{T} \int_{\Omega} \mu u^{\kappa} \leq \liminf _{\varepsilon=\varepsilon_{k} \sea 0}  \int_{0}^{T} \int_{\Omega}  \mu u_{\varepsilon}^{\kappa}
\end{equation*}
and because of \eqref{e229}, \eqref{eq:conv_init} and \eqref{e228},
taking the limes superior on both sides of (\ref{e42}) yields (\ref{e24}).

Moreover, for and $\varphi \in C_c^\infty(\Ombar \times [0, \infty))$,
we multiply the second equation in (\ref{e210}) by $\varphi$ and integrate by parts to obtain
\begin{equation*}
-\int^\infty_0\int_{\Omega}v_\varepsilon\varphi_t-\int_\Omega v_{\eps 0}\varphi(\cdot, 0)=-\int^\infty_0\int_{\Omega}\nabla v_\varepsilon\cdot\nabla\varphi-\int^\infty_0\int_{\Omega}v_\varepsilon\varphi+\int^\infty_0\int_{\Omega}u_\varepsilon\varphi
\qquad \text{for all $\eps \in (0, 1)$},
\end{equation*}
which according to \eqref{e231}, \eqref{eq:conv_init}, \eqref{e232} and \eqref{e228} implies (\ref{e26}).

Let us now proceed to prove \eqref{e25}.
For nonnegative $\varphi\in C_c^\infty(\ol{\Omega}\times[0,\infty))$, testing the first equation in (\ref{e210}) with $\frac{\varphi}{u_{\varepsilon}+1}$ leads to
\begin{equation}\label{e244}
\begin{aligned} I_{1}(\varepsilon)&\defs \int_{0}^{\infty} \int_{\Omega}\left|\nabla \ln \left(u_{\varepsilon}+1\right)\right|^{2} \varphi \\
&=-\int_{0}^{\infty} \int_{\Omega} \ln \left(u_{\varepsilon}+1\right) \varphi_{t}-\int_{\Omega} \ln \left(u_{0}+1\right) \varphi(\cdot, 0)
+\int_{0}^{\infty} \int_{\Omega} \nabla \ln \left(u_{\varepsilon}+1\right) \cdot \nabla \varphi \\
&\pe+\int_{0}^{\infty} \int_{\Omega} \frac{f_{\varepsilon}\left(u_{\varepsilon}\right)}{u_{\varepsilon}+1}\left(\nabla \ln \left(u_{\varepsilon}+1\right) \cdot \nabla v_{\varepsilon}\right) \varphi -\int_{0}^{\infty} \int_{\Omega} \frac{f_{\varepsilon}\left(u_{\varepsilon}\right)}{u_{\varepsilon}+1} \nabla v_{\varepsilon} \cdot \nabla \varphi \\
&\pe-\int_{0}^{\infty} \int_{\Omega}\frac{\lambda u_\varepsilon}{u_\varepsilon+1}\varphi+\int_{0}^{\infty} \int_{\Omega}\frac{\mu u_\varepsilon^\kappa}{u_\varepsilon+1}\varphi\\
&\sfed I_{2}(\varepsilon)+I_{3}(\varepsilon)+I_{4}(\varepsilon)+I_{5}(\varepsilon)+I_{6}(\varepsilon) +I_{7}(\varepsilon)+I_{8}(\varepsilon)
\qquad\text{for all $\eps \in (0, 1)$}.
\end{aligned}
\end{equation}
We choose $T>0$ large enough such that $\supp \varphi \in \Ombar \times [0, T]$.
Since $\ln \left(u_{\varepsilon}+1\right) \rightharpoonup \ln (u+1)$ in $L^2((0,T);W^{1,2}(\Omega))$ as $\varepsilon=\varepsilon_k\searrow0$ by (\ref{e230}), we have
\begin{equation}\label{e245}
I_2(\varepsilon)\to-\int_{0}^{\infty} \int_{\Omega} \ln (u+1) \varphi_{t}\ \ \text{and}\ \ I_4(\varepsilon)\to\int_{0}^{\infty} \int_{\Omega} \nabla \ln (u+1) \cdot \nabla \varphi
\qquad \text{as $\eps = \eps_k \sea 0$}
\end{equation}
and
\begin{equation}\label{e246}
\int_{0}^{\infty} \int_{\Omega}|\nabla \ln (u+1)|^{2} \varphi \leq \liminf _{\varepsilon=\varepsilon_{k} \sea 0} I_{1}(\varepsilon).
\end{equation}
Moreover, because $\nabla v_\varepsilon\to\nabla v$ in $L^2(\Omega\times(0,T))$ as $\varepsilon=\varepsilon_k\searrow0$ by Lemma \ref{lem2.9}, since $0\leq\frac{f_{\varepsilon}(u_{\varepsilon})}{u_{\varepsilon}+1}\leq1$ in $\Omega \times (0, T)$ for all $\varepsilon\in(0,1)$ and as $\frac{f_{\varepsilon}(u_{\varepsilon})}{u_{\varepsilon}+1}\to\frac{u}{u+1}$ a.e.\ in $\Omega\times(0,T)$ as $\varepsilon=\varepsilon_k\searrow0$ by \eqref{e228} and the definition of $f_\eps$, we have
\begin{align*}
        \left\| \frac{f_\eps(\ue)}{\ue+1} \nabla \ve - \frac{u}{u+1} \nabla v \right\|_{\leb2}
  &\le  \left\| \frac{f_\eps(\ue)}{\ue+1} ( \nabla \ve - \nabla v ) \right\|_{\leb2}
        + \left\| \left(\frac{f_\eps(\ue)}{\ue+1} - \frac{u}{u+1} \right) \nabla v \right\|_{\leb2}
   \ra  0
\end{align*}
as $\varepsilon=\varepsilon_k\searrow0$,
which together with \eqref{e230} implies
\begin{equation}
I_{5}(\varepsilon) \rightarrow -\int_{0}^{\infty} \int_{\Omega} \frac{u}{u+1}(\nabla \ln (u+1) \cdot \nabla v) \varphi \label{e46}
\qquad\text{as $\varepsilon=\varepsilon_k\searrow0$}.
\end{equation}
Noting that $\frac{\mu \ue^\kappa}{\ue+1} \ra \frac{\mu u^\kappa}{u+1}$ in $L^1(\Omega \times (0, T))$ as $\eps = \eps_k \sea 0$
by \eqref{e218}, \eqref{e228} and Vitali's theorem, we obtain by analogous arguments
\begin{align}
I_{6}(\varepsilon) \rightarrow \int_{0}^{\infty} \int_{\Omega} \frac{u}{u+1} \nabla v \cdot \nabla \varphi, \quad 
I_{7}(\varepsilon) \rightarrow -\int_{0}^{\infty} \int_{\Omega} \frac{\lambda u}{u+1}\varphi \quad \text{and} \quad
I_{8}(\varepsilon) \rightarrow \int_{0}^{\infty} \int_{\Omega} \frac{\mu u^\kappa}{u+1} \varphi \label{e49},
\end{align}
as $\varepsilon=\varepsilon_k\searrow0$.
Finally, (\ref{e25}) follows directly from~\eqref{e244}--\eqref{e49}.
\end{proof}

\small
\section*{Acknowledgments}
The first author has been supported by China Scholarship Council (No.\ 201906090124),  and in part by National Natural Science Foundation of China (Nos.\ 11671079, 11701290, 11601127 and 11171063), and the Natural Science Foundation of Jiangsu Province (No.\ BK20170896). The second author is partially supported by the German Academic Scholarship Foundation and by the Deutsche Forschungsgemeinschaft within the project \emph{Emergence of structures and advantages in cross-diffusion systems}, project number 411007140.

\footnotesize 

\newcommand{\noopsort}[1]{}

\end{document}